\documentclass[12pt]{amsart}
\usepackage{amssymb,amsmath,amsthm}
\usepackage{tikz}

\newcommand{\old}[1]{}
\theoremstyle{plain}
\newtheorem{thm}{Theorem}[section]
\newtheorem{lem}[thm]{Lemma}

\newtheorem{cor}[thm]{Corollary}

\newtheorem{prop}[thm]{Proposition}
\theoremstyle{definition}
\newtheorem{defn}[thm]{Definition}

\newtheorem{ex}[thm]{Example}
\newtheorem{rk}[thm]{Remark}

\title{Diameter bound  for  facet-ridge incidence  graphs of geometric lattices}

\author{Patricia Hersh} 
\address{Department of Mathematics, University of Oregon, Eugene, OR 97403}
\email{plhersh@uoregon.edu}

\author{John Machacek}
\address{Department of Mathematics, University of Oregon, Eugene, OR 97403}
\email{johnmach@uoregon.edu}

\thanks{
 This work was supported by 
NSF grants 
DMS-1953931 and NSF-RTG grant DMS-2039316.
}


\begin{document}

\begin{abstract}
This paper proves that the facet-ridge incidence graph of the order complex of any  finite geometric lattice of rank $r$ has diameter at most ${r \choose 2}$.  A key ingredient is the well-known  fact that every ordering of the atoms of any finite geometric lattice gives rise  to a lexicographic shelling of its order complex.  The paper also gives results that provide  some evidence that  this bound ought to be  sharp as well as examples indicating that the question of sharpness is quite subtle.
\end{abstract}

\maketitle

 \\

\section{Introduction}\label{intro-section}

In this short paper, we prove one direction of  a conjecture of Jose Samper.  He conjectured (personal communication) that the diameter of the facet-ridge graph 
 of any finite  geometric lattice of rank $r$ is exactly ${r \choose 2}$.  We prove that this upper bound holds.   More specifically, for any pair of maximal chains in a finite geometric lattice of rank $r$, we construct a path of length at most ${r\choose 2}$ in the facet-ridge incidence graph  from one maximal chain   to the other.    
 We also give results that we be believe are steps towards  proving sharpness  of this bound and in fact proving  something stronger than sharpness. 
  That is, we prove that each vertex in our graph has another vertex at distance exactly 
 ${r\choose 2}$ from it in a spanning directed  subgraph where we would  expect to find the shortest paths between this  pairs of vertices. 

A key ingredient in our proof of the diameter upper bound   is the result  from \cite{BJ1} (see also \cite{DH})  that every ordering on the atoms of a geometric lattice gives rise to a so-called minimal labeling that is an EL-labeling.  
Given any two maximal chains  
 in a finite geometric lattice, this result allows us 
 to choose an EL-labeling  such  that  
 one of these maximal chains  comes first in lexicographic order.  In particular, this reduces the diameter upper bound
 question to showing there is a path of length at most ${r\choose 2}$ from any maximal chain to the lexicographically earliest one.   Each step in the path that we construct  transforms a single descent in the associated label sequence 
 into an ascent.


We associate 
a sequence of adjacent transpositions in the symmetric group to each monotone  path in the graph, namely to  each path  that is monotonic   with respect to lexicographic order. 
 We prove for any maximal  chain  that there exists   one 
 such   sequence of adjacent transpositions leading to the lexicographically earliest maximal chain
 which is a reduced expression for some element  in the symmetric group $S_r$ (see Section ~\ref{bg-section} for definitions).  This  allows us to conclude that the upper bound is at most the length of the longest element in the symmetric group.  The desired upper bound then follows from the fact that 
 Coxeter-theoretic length of each element in the symmetric group $S_r$ is at most  ${r\choose 2}$.   

It may be helpful to consider the very special case in which our geometric lattice is  the poset $B_r$  of subsets of $\{ 1,2,\dots ,r\} $ ordered by containment.  Any minimal labeling (see Definition ~\ref{min-def}) for $B_r$  will label the  $r!$ maximal chains with the various permutations in $S_r$, with each permutation labeling exactly one maximal chain.   In this case the diameter  of the facet-ridge incidence graph of the order complex is exactly ${r\choose 2}$, as we explain next.   The length of the shortest path between two permutations $\sigma , \tau \in S_r$  is the number of adjacent transpositions needed to transform the one permutation to the other.  This is exactly the length of the permutation $\sigma^{-1}\circ \tau $.  It is  exactly the number of pairs of  elements of $\{ 1,2,\dots ,r\}$ that are in order in one permutation but out of order in the order.  There are at most ${r\choose 2}$ such pairs,  giving the desired upper bound.  In the case that $\sigma $ is the reversal of $\tau $,  there are  exactly ${r\choose 2}$ such pairs, demonstrating sharpness of the bound in this case.  

It is not hard to show that every geometric lattice of rank $r$ contains a copy of $B_r$ within it.  Indeed  we effectively verify this claim  when we show for the ascending label sequence that the reversal of this sequence also occurs on some maximal chain and also show that each label sequence obtained from this decreasing sequence by applying a series of adjacent tranpositions also occurs.  We provide evidence  that the distance between the reversal of the ascending chain and the ascending chain itself is exactly ${r\choose 2}$.  This distance of ${r\choose 2}$  would imply sharpness of our upper bound on the diameter of the facet ridge incidence graph  and hence would prove Samper's conjecture.


An upper bound on the diameter of facet ridge incidence graphs  for all  pure, strongly connected  simplicial complexes is given in~\cite{CS}.  This bound is far  larger than our  upper bound, as one might expect  in light of the matroid exchange axiom and consequent high level of connectivity in  the class of pure strongly connected 
 simplicial complexes  that we are considering. 
A particularly important setting  where diameters of  facet-ridge incidence graphs arise is in the work of Francisco Santos disproving  the Hirsch Conjecture (see \cite{Sa}); his counterexamples involve the 1-skeleta of polytopes, but to analyze them he first takes their polytopal duals where the 1-skeleta of the original polytopes become exactly the facet-ridge incidence graphs of their dual polytopes.



\section{background}\label{bg-section}

We begin with a  quick review  of terminology and prior results  that will be  needed for the proof of our main result and for understanding the statements of our results.

\begin{defn}
The {\bf facet-ridge incidence graph} of a  simplicial complex $\Delta $  whose maximal faces all have the same dimension  is the graph whose vertices are the facets (i.e. maximal faces) of $\Delta $ and whose edges are the pairs of facets sharing a codimension one face (known as a ``ridge'').
\end{defn}

\begin{defn}
The {\bf distance}  $d(u,v)$ between a pair of vertices $u,v$ in a graph $G$ is the minimal number of edges in any path from $u$ to $v$.
The {\bf diameter} of a graph $G$ is the largest distance between any pair of vertices of $G$.   Thus, $diam (G) = \max_{u,v\in V(G)} d(u,v)$.
\end{defn}

\begin{defn}
A {\bf cover relation } in a partially ordered set (poset) is an order relation $u < v$ in which $u\le z\le v$ implies $z=u$ or $z=v$.  Such an order relation is denoted $u\lessdot v$.  

A  partially ordered set  $P$ is {\bf bounded} if it has a unique minimal element $\hat{0}$ and unique maximal element $\hat{1}$.  
A bounded poset 
 is {\bf graded} if all maximal chains 
$\hat{0}\lessdot x_1 \lessdot \cdots \lessdot x_{r-1}\lessdot \hat{1}$   in $P$ 
have this  same number $r$  of cover relations as each other.  In this case, this number $r$  is called the {\bf rank}
of the poset, and we say that the element $x_i$ in a maximal chain $\hat{0}\lessdot x_1 \lessdot \cdots \lessdot x_i \lessdot \cdots x_{r-1}\lessdot \hat{1}$ has rank $i$.  
\end{defn}

\begin{defn}
Any element covering $\hat{0}$ in a finite lattice is called an {\bf atom}.  A finite  lattice  $L$ is called {\bf atomic} if each element of $L$ is expressible as a join, namely a least upper
bound, of a set of atoms of $L$.  It is {\bf semimodular} if $x\lessdot y$ and $x\lessdot y'$ in $L$ implies the existence of $z\in L$ with $y\lessdot z$ and $y'\lessdot z$.
\end{defn}

\begin{defn}
A  lattice is {\bf geometric} if and only if it is finite, atomic and semimodular.    Alternatively, one may take as the definition that a finite  
lattice is geometric if and only if it is the lattice of flats of a simple matroid.  We refer readers to \cite{Bj2} 
 for basic notions and definitions  of matroid theory.
\end{defn}


\begin{rk}
For finite lattices  semimodularity implies gradedness, as discussed in \cite{st}.  Thus, geometric lattices are graded.
\end{rk}

Bj\"orner introduced the notion of EL-labeling, reviewed next, in \cite{BJ1}.  He proved that any finite bounded poset with an EL-labeling was shellable via a type of shelling he called an EL-shelling.   An {\bf edge labeling} of a finite partially ordered set $P$  is a function $\lambda :P\rightarrow \Lambda  $ assigning to each cover relation $u  \lessdot v$ 
in $P$ an element $\lambda (u,v)$  of a partially  ordered set $\Lambda $.  In our setting,  $\Lambda $ will always be a totally ordered set.    Notice that any such edge labeling  $\lambda $ may be used to  assign a 
label sequence $(\lambda (u_0,u_1), \lambda (u_1,u_2),\dots ,\lambda (u_{r-1},u_r))$ to each maximal chain $u_0\lessdot u_1\lessdot \cdots \lessdot u_r$ in $P$.  

An ordered pair $(\lambda (u_{i-1},u_i), \lambda (u_i,u_{i+1}))$ of labels  on a pair  $u_{i-1}\lessdot  u_i $ and $u_i  \lessdot  u_{i+1}$ of consecutive cover relations 
is called an {\bf ascent}  
 if  $\lambda (u_{i-1},u_i) \le \lambda (u_i,u_{i+1})$.  This ordered pair of labels  is called a {\bf descent} otherwise.  

\begin{defn}
An edge labeling $\lambda $ of a finite bounded poset $P$ is an {\bf EL-labeling} if for each $u<v$ in $P$ both of the following conditions hold:
\begin{enumerate}
\item
There is a unique saturated chain $u\lessdot  u_1\lessdot  u_2\lessdot  \cdots \lessdot u_k\lessdot v$ from $u$ to $v$ (called the ``ascending chain'' from $u$ to $v$) 
satisfying $\lambda (u,u_1)\le \lambda (u_2,u_3)\le \cdots \le \lambda (u_k,v)$.
\item
This saturated chain has label sequence which is strictly smaller in lexicographic order than the label sequence of every other saturated chain from $u$ to $v$.
\end{enumerate}
\end{defn}


Next we describe a family of EL-labelings for geometric lattices,  introduced by Bj\"orner in \cite{BJ1},  that
will be critical to our work.   

\begin{defn}[Bj\"orner]\label{min-def}
Let $A$ denote the atoms of a  geometric lattice $L$.  For each $u\in L$ 
let $A(u) = \{ a\in A| a\le u\} $. 

Given any ordering $a_1,\dots ,a_n$ on the elements of  $A$, 
define the  {\bf minimal labeling} $\lambda $ on  
 $L$ induced by this atom  ordering
 to be the edge labeling with label set $A$ that is  given by  $\lambda (u,v) = \min_{a\in A(v)\setminus A(u)} a$ with ordering $a_1<a_2<\cdots <a_n$ on our set of  edge labels.
\end{defn}


The following important  related result of Bj\"orner also   appears in  \cite{BJ1}. 

\begin{thm}[Bj\"orner]\label{any-atom-ordering}
The minimal labeling induced by any total ordering on the atoms of a geometric lattice $L$ is an EL-labeling.
\end{thm}


\begin{defn}
The {\bf order complex} of a finite partially ordered set $P$  is the abstract simplicial complex, denoted $\Delta (P)$, whose $i$-dimensional 
faces are the chains $v_0 < v_1 < \cdots < v_i$ of $i+1$ 
comparable elements in $P$.
\end{defn}

\begin{defn}
A {\bf reduced expression} for a permutation $\pi $  in the symmetric group $S_r$ is a way to write $\pi $ as a product $s_{i_1}\cdots s_{i_k}$ of adjacent transpositions
$s_i = (i,i+1)$ with $k$ as small as possible.  In this case, $k$ is called the {\bf length} of $\pi $.    This length of $\pi $ also equals the number of {\bf inversion pairs} in $\pi $, namely the number of ordered pairs $(i,j)$ with $i<j$ and $\pi (i)>\pi (j)$.  

The {\bf longest element} in $S_r$ is the unique permutation having maximal length, namely the reverse permutation $(r,r-1,\dots ,2,1)$.  This  has length ${r\choose 2}$.
\end{defn}

\begin{defn}\label{reflection-sequence}
Given any  expression $s_{i_1}\cdots s_{i_d}$ in $S_r$, we associate to this the sequence 
$$(s_{i_1},s_{i_1}s_{i_2}s_{i_1},\dots , s_{i_1}s_{i_2}\cdots s_{i_{d-1}}s_{i_d}s_{i_{d-1}}\cdots s_{i_1})$$ of reflections in $S_r$.  
This is discussed more generally for Coxeter groups  in Chapter 1 of \cite{BB}.
\end{defn}

Next we give a quick proof of the following result whose consequence in the case of the symmetric group, discussed next, will be critical to our work:

\begin{lem}
An expression in a finite Coxeter group  $W$ is reduced if and only if its associated sequence of reflections (as in Definition ~\ref{reflection-sequence})  has no repeats. 
\end{lem}
\begin{proof}
Lemma 1.3.1 in \cite{BB} proves one direction, namely that reduced expressions have no repeats in their associated sequences of reflections.  

For the other direction, suppose that $s_{i_1}\cdots s_{i_d}$ is nonreduced.  By the deletion property for Coxeter groups, this implies that 
$$s_{i_1}\cdots s_{i_d} = s_{i_1}\cdots \hat{s}_{i_j}\cdots \hat{s}_{i_k}\cdots s_{i_d}$$ for some $1\le j < k\le d$.    But then 
$s_{i_j}\cdots s_{i_k} = s_{i_{j+1}}\cdots s_{i_{k-1}}=w$ for some $w\in W$, which implies $w^{-1} = s_{i_{k-1}}\cdots s_{i_{j+1}}$  and hence implies 
$$(s_{i_j}\cdots s_{i_{k-1}}s_{i_k})(s_{i_{k-1}}\cdots s_{i_{j+1}})s_{i_j}  = ww^{-1}s_{i_j} = s_{i_j}.$$  But then multiplying both sides on the left by $s_{i_1}\cdots s_{i_{j-1}}$ and both sides on the right by $s_{i_{j-1}}\cdots s_{i_1}$
yields $$(s_{i_1}\cdots  s_{i_{j-1}})(s_{i_j} \cdots s_{i_{k-1}}s_{i_k}s_{i_{k-1}}\cdots  s_{i_j}))(s_{i_{j-1}}\cdots s_{i_1}) = (s_{i_1}\cdots s_{i_{j-1}})s_{i_j}(s_{i_{j-1}}\cdots s_{i_1}).$$  
This exhibits that two reflections in the sequence of reflections associated to $s_{i_1}\cdots s_{i_d}$ are equal to each other.  
\end{proof}




\begin{defn}
A {\bf wiring diagram} is a family of $r$ piecewise-straight lines called wires,  
which can be viewed as graphs of $r$ continuous piecewise-linear functions defined on the same interval from left to right. The wires are labeled $1,2,\dots, r$. We draw wiring diagrams as shown in Figure~\ref{fig:red} and Figure~\ref{fig:nonred}. In particular, we draw the graphs of the piecewise-linear functions from left to right. On the left we label wires $1,2,\dots, r$ from top to bottom. The slope of pieces of each graph will only take values from $\{-1, 0, 1\}$. Furthermore, the interval is split into equal sized discrete time intervals, and on each such time interval a wire with slope $-1$ crosses a wire with slope $1$ while all other wires have slope $0$.
\end{defn}

Wiring diagrams with wires labled $1,2,\dots, r$ and expressions for elements in $S_r$ using the generators $\{s_1, s_2, \dots, s_{r-1}\}$ are in bijection as follows.
The $i$th highest wire crossing the  $(i+1)$-th highest wire corresponds to the generator $s_i$.
We may consider the crossings from left to right in a wiring diagram  to get an expression.
Conversely, given an expression we may always construct a wiring diagram with the proscribed crossings. 
Figures~\ref{fig:red} and ~\ref{fig:nonred} show wiring diagrams along with the expressions with which  they are  in bijection. 

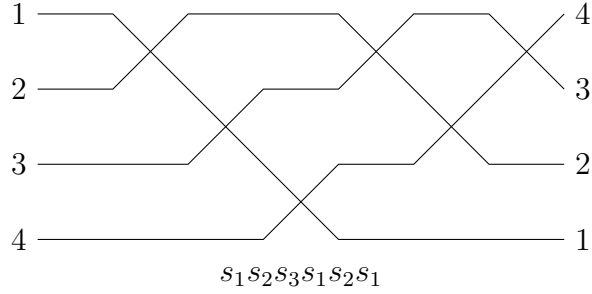
\begin{figure}
\centering
\begin{tikzpicture}
\draw (0,9) -- (1,9) -- (2,8) -- (3,7) -- (4,6) -- (5,6) -- (6,6) -- (7,6);
\draw (0,8) -- (1,8) -- (2,9) -- (3,9) -- (4,9) -- (5,8) -- (6,7) -- (7,7);
\draw (0,7) -- (1,7) -- (2,7) -- (3,8) -- (4,8) -- (5,9) -- (6,9) -- (7,8);
\draw (0,6) -- (1,6) -- (2,6) -- (3,6) -- (4,7) --(5,7) -- (6,8) -- (7,9);
  \node at (-0.25,9) {$1$};
    \node at (-0.25,8) {$2$};
  \node at (-0.25,7) {$3$};
  \node at (-0.25,6) {$4$};
    \node at (7.25,9) {$4$};
    \node at (7.25,8) {$3$};
  \node at (7.25,7) {$2$};
  \node at (7.25,6) {$1$};
  \node at (3.5,5.5) {$s_1 s_2 s_3 s_1 s_2 s_1$};
\end{tikzpicture}
\caption{A wiring diagram corresponding to the reduced expression $s_1 s_2 s_3 s_1 s_2 s_1$ for the longest element in $S_4$.  
Each pair of wires crosses exactly once.}
\label{fig:red}
\end{figure}

\begin{figure}
\centering
\begin{tikzpicture}
  \draw (0,4) -- (1,4) -- (2,3) -- (3,2) -- (4,2) -- (5,1) -- (6,1);
  \draw (0,3) -- (1,3) -- (2,4) -- (3,4) -- (4,3) -- (5,3) -- (6,4);
  \draw (0,2) -- (1,2) -- (2,2) -- (3,3) -- (4,4) -- (5,4) -- (6,3);
  \draw (0,1) -- (2,1) -- (3,1) -- (4,1) -- (5,2) -- (6,2);
  \node at (-0.25,4) {$1$};
    \node at (-0.25,3) {$2$};
  \node at (-0.25,2) {$3$};
  \node at (-0.25,1) {$4$};
    \node at (6.25,4) {$2$};
    \node at (6.25,3) {$3$};
  \node at (6.25,2) {$4$};
  \node at (6.25,1) {$1$};
  \node at (3,0.5) {$s_1 s_2 s_1 s_3 s_1$};
\end{tikzpicture}

\caption{A wiring diagram which corresponds to the  expression $s_1 s_2 s_1 s_3 s_1$ which is not reduced.  
  The pair of wires labeled by $2$ and $3$ cross twice.}
\label{fig:nonred}
\end{figure}
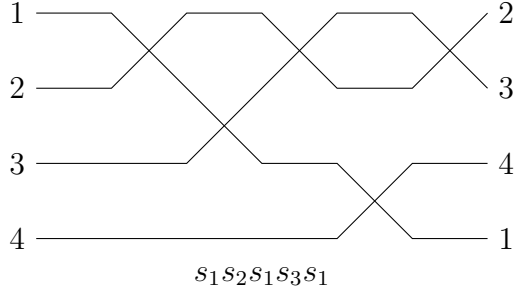

\begin{defn}
When a pair of wires cross each other twice in a wiring diagram, we call this a {\bf double crossing}.
\end{defn}

\begin{cor}\label{no-double}
An expression  for a permutation in $S_r$ is reduced if and only if its wiring diagram has no double crossings.
\end{cor}
\begin{proof}
Given an expression $s_{i_1}\cdots s_{i_d}$  for a permutation $\pi \in S_r$, notice that the associated sequence of reflections can be described in terms of the wiring diagram 
given by the expression for $\pi $ 
as follows.       The reflection $s_{i_1}s_{i_2}\cdots s_{i_j}\cdots s_{i_2}s_{i_1} = (l,m)$ records  the pair $l,m$ of wires in the wiring diagram that  cross each other at the 
step given by the simple reflection $s_{i_j}$.
 Thus, having the same  pair of wires cross  each other twice is equivalent to encountering the same reflection twice in the  sequence of reflections associated to the 
expression $s_{i_1}\cdots s_{i_d}$  for $\pi $.
\end{proof}

\section{Diameter upper bound}


\begin{defn}\label{G-lex}
Any EL-labeling  $\lambda $ of any finite, graded poset $P$ gives rise to an orientation 
on a spanning subgraph 
of the facet-ridge graph of $\Delta (P)$.   Denote this oriented graph by  $G_{lex}$,  with $G_{lex}$ constructed as follows.  
Given any maximal chain $M$ in $P$  and any descent in  $\lambda (M)$, we obtain another maximal chain $M'$ by deleting the element  of $M$ located at this  descent and letting  $M'$ be the unique maximal chain which has this descent replaced by an ascent and which otherwise coincides with $M$.  
In this case, the two vertices  given by $M$ and $M'$ in the facet-ridge incidence  graph are connected by an edge in $G_{lex}$, with this edge oriented 
from $M$ to $M'$.   All of the directed edges in $G_{lex}$ arise this way.
\end{defn}

\begin{defn}
Whenever a directed edge in $G_{lex}$ replaces a descent at rank $i$ in a maximal chain $M$ with an ascent to obtain a new maximal chain $M'$, 
let us say that $M' = T_i(M)$.  Define $T_{i_1}\cdots T_{i_d}(M) $ inductively to be $T_{i_1}\cdots T_{i_{d-1}}(T_{i_d}(M))$.  
\end{defn}







\begin{lem}\label{assoc-an-expression}
Given any geometric lattice $L$ of rank $r$, any minimal labeling  $\lambda $ for $L$, and  any maximal chain $M$ in $L$, there exists  
$\pi \in S_r$ and a 
reduced expression 
$s_{i_1}\cdots s_{i_d}$ for $\pi $
such that $T_{i_1}\cdots T_{i_d}(M)$ is the unique ascending chain from $\hat{0}$ to $\hat{1}$ in $L$ with respect to $\lambda $.
\end{lem}  

\begin{proof}
We will  construct a  
 sequence  of moves  $T_{i_d}, T_{i_{d-1}},\dots T_{i_1}$ 
  to apply to   
  $M$, and then we will  prove that $s_{i_1}\cdots s_{i_d}$ is a reduced expression for a permutation in $S_r$. 
   
Let $\lambda_1$ be the  smallest label appearing anywhere within the label sequence for $M$.  By virtue of how minimal labelings are defined, $\lambda_1$ must be the smallest
 atom in our set $A$ of atoms of $L$ (sometimes called the ground set).  
Unless $\lambda_1$  is the lowest label on $M$, it forms a descent with the label  $\mu $  just below it (i.e. on the cover relation just below the cover relation labeled 
$\lambda_1$).    When there is such a $\mu $ just below $\lambda_1$, consider the positive integer
$i_1\ge 1$ such that  $\mu $  is the label  on the cover relation of
$M$  between ranks  $i_1-1$ and $i_1$. 
 In this case, 
 we begin by  applying  the operator  $T_{i_1}$.  Observe that by definition of minimal labeling, by our choice of $\lambda_1$ as the overall smallest atom in $L$,  by the fact that
 geometric lattices are exactly the lattices of flats of simple matroids, 
 and by properties of lattices of flats of simple matroids (namely the notions of rank and of flat), 
 it follows that   $T_{i_1}(M)$ has the label $\lambda_1$ shifted downward one place in the label sequence 
 to between ranks $i_1-1$ and $i_1$ and  that $T_{i_1}(M)$ has a label $\mu '$  that is larger than $\lambda_1$  appearing just  above 
 $\lambda_1$ in $\lambda (T_{i_1}(M))$.  This label  $\mu '$ 
 does  not necessarily equal   $\mu $, 
 as  
 $\mu '$ 
  could be smaller than $\mu $.     Next we apply the operator 
    $T_{i_1-1}$, then $T_{i_1-2}$, and continue 
    applying progressively lower operators  in this manner up to and including applying $T_1$.  Since $\lambda_1$ is the 
    smallest element of our ground set, each  step  moves the label  $\lambda_1$ downward one position,  
    and each step puts a label that is larger than $\lambda_1$ just above $\lambda_1$.
 At the end of this process, 
  we will have a maximal chain  $T_1 T_2\cdots T_{i_1}(M)$ which has the label $\lambda_1$ on   its lowest cover relation. 
 Let 
     $M_2 = T_1 T_2\cdots T_{i_1}(M)$.   In the case where $M$ already  had $\lambda_1$ as its lowest label, let $M_2 = M$.   
      In either case, let $y_1$ denote the element of rank 1 in $M_2$. 

Let  $\lambda_2$ denote the second smallest label in  $\lambda (M_2)$.  Note that $\lambda_2$   is 
smaller than every  atom of $L$  not in $A(y_1)$ due to the fact that the smallest atom not in $A(y_1)$ must appear as a label on every saturated chain from $y_1$ to $\hat{1}$ by 
virtue of how minimal labels are defined.   Additionally, observe that $\lambda_2$ is the smallest label appearing on any cover relation $y_1\lessdot z$ for $z\in L$ and can only 
appear on one such cover relation; to see this, observe that  the interval  $[y_1,\hat{1}]$ in $L$ is itself a geometric lattice (due to being the lattice of flats of a simple matroid)  and 
that it has a minimal labeling coinciding  with the restriction of   our minimal labeling $\lambda $ 
to $[y_1,\hat{1}]$.  Notice that $\lambda_2$ is the smallest label in this minimal labeling for $[y_1,\hat{1}]$, 
allowing us to think of the aforementioned  $z$ covering $y_1$ with $\lambda (y_1,z) = \lambda_2$ as being the smallest atom $\lambda_2$ for $[y_1,\hat{1}]$.  Thus, we
may  apply 
the reasoning above regarding $\lambda_1$ to the downward shifting of the label  $\lambda_2$, as described next.  
 If $\lambda_2$ is not  immediately  above 
$\lambda_1$ in our label sequence  $\lambda (M_2)$, 
then denote by  $\mu_2$  the label just below $\lambda_2$ in $\lambda (M_2)$, and let $i_2\ge 1$ be the rank of the element  in $M_2$  appearing 
between the labels $\mu_2$ and 
$\lambda_2$.  Similarly to with $\lambda_1$, we next   apply the series of operators  $T_{i_2}$, then $T_{i_2-1} $, and continue proceeding  downward in this 
manner up to and including 
applying  $T_2$.  By the same reasoning as used above for $\lambda_1$, this sequence of moves  has the effect of 
shifting the label $\lambda_2$ downward  to the second lowest position in the label sequence in the resulting maximal chain 
$(T_2 T_3\cdots T_{i_2})(T_1T_2\cdots T_{i_1})(M)$, a maximal chain which we call 
$M_3$.   If $M_2$ already  has $\lambda_2$ as its lowest label within $[y_1,\hat{1}]$, then let $M_3=M_2$.
Let $y_2$ be the element of rank 2 in $M_3$, and observe that $y_1$ is the element of rank 1 in $M_3$.  

Next we similarly shift  
  the  third smallest label of $\lambda (M_3)$, denoted $\lambda_3$,  downward in this same manner  until  it is on a cover relation just above $y_2$.   Since 
$\lambda_3$ is the smallest atom not in $A(y_2)$, it will continue to be one of the labels in our label sequence  at each step 
as we shift it downward to the position just above $\lambda_2$.
Note that  $\lambda_3$  will  be  the third 
smallest  label in the label sequence on the  resulting maximal chain, denoted $M_4$, since  $\lambda_3$ 
is smaller than all labels appearing above it in  $\lambda (M_4)$ and is larger 
than all  labels appearing   below it; 
here  we use the fact  that 
each time we shift $\lambda_3$ downward past a label, the label that appears immediately above it after the shift must be larger than 
$\lambda_3$.  
Once we are done shifting $\lambda_3$, we likewise continue shifting  progressively larger and larger labels downward 
 in this manner until eventually  reaching  
a  maximal chain  whose label sequence is an ascending chain.  
That is, we apply 
to $M$ a series of  operators 
$\cdots (T_3\cdots T_{i_3-1}T_{i_3})(T_2\cdots T_{i_2-1}T_{i_2})(T_1\cdots T_{i_1-1}T_{i_1})$, applying these operators  from right to left.

Next we will  show  that   the expression 
$\cdots (s_3\cdots s_{i_3-1}s_{i_3})(s_2\cdots s_{i_2-1}s_{i_2})(s_1\cdots s_{i_1-1}s_{i_1})$ 
with the same series of indices  as $\cdots (T_3\cdots T_{i_3-1}T_{i_3})(T_2\cdots T_{i_2-1}T_{i_2})(T_1\cdots T_{i_1-1}T_{i_1})$ 
is a reduced expression for some $\pi\in S_r$. 
We  do this by showing   that the wiring diagram for this sequence of adjacent transpositions  
 has no double crossings.  That will allow us to apply Corollary ~\ref{no-double} to deduce that this is a reduced expression.   
 Since we are applying operators from right to left, one needs  to read the corresponding wiring diagram from right to left as well.
 
  First notice that the wire representing 
 the smallest label  $\lambda_1$  
 crosses distinct wires  at the various steps in which it crosses wires in the course of   shifting $\lambda_1$   downward  to  
  below all of the  other edge  labels.  
 Once this wire has shifted to below all other wires,   it  never crosses any further wires, due to $\lambda_1$ staying below all other labels during the remainder of our label shifting process.  Thus, this first wire to shift downward does not double cross any other wires.
 By the same reasoning, 
  the wire representing 
  the label $\lambda_2$  
  also crosses distinct wires as it shifts downward to the second lowest position,   and it does not cross any wires subsequent to that. 
  This wire representing $\lambda_2$ also does not double cross 
  the wire representing $\lambda_1$, as is  verified by noting that we already showed that the  $\lambda_1$ wire  does not double cross the   $\lambda_2$ wire.
  We apply this same reasoning repeatedly,  thereby showing  that the  wire representing each $\lambda_i$  in 
  turn only crosses each of the other wires at most once as the wire  representing $\lambda_i$ 
  moves downward and also does not double cross any wire  that crossed it prior to its downward shifting.  
  Thus, no two wires double cross, completing our proof.
\end{proof}



\begin{cor}\label{dist-to-ascending}
Given any geometric lattice $L$ of rank $r$ and any minimal labeling for $L$,  every maximal chain in $L$ is within 
graph-theoretic distance at most ${r \choose 2}$ of the unique  ascending chain of $L$,  with this distance calculated in the facet-ridge incidence graph of 
$\Delta (L)$.  
\end{cor}

\begin{proof}
We use that the longest  element  in $S_r$  has length ${r\choose 2}$ to deduce  the desired upper bound from Lemma ~\ref{assoc-an-expression}, since the length of a 
permutation is the length of each of its reduced expressions.  
\end{proof}

In light of the proof of  Lemma \ref{assoc-an-expression}, one might  hope  that the  sequence of adjacent transpositions associated to 
 any directed path in $G_{lex}$  gives rise to a reduced expression for a permutation in $S_r$.  One might further  hope  that  all of these monotone paths from a fixed
 maximal chain $M$ to  a fixed 
maximal chain $M'$   give rise to  reduced expressions  for the same permutation. 
 However, these statements are  not  always true, as the next  example shows. 

\begin{ex}
Consider the lattice of  flats for the uniform matroid $M_{2,4}$.  In other words, consider 
the poset of subsets of $\{ 1,2,3,4\}$ of size not equal to 3 ordered by containment.  
Take the atom ordering $a_1 < a_2 < a_3 < a_4$ where $a_i$ denotes  the subset $\{ i\} $.  Let $\lambda $ be the resulting minimal labeling,  using 
the label $i$ for the atom $a_i$ with label ordering $1<2<3<4$.  

 Let $M$ denote  the maximal chain $\emptyset < \{ 4\} < \{ 3,4\} < \{ 1,2,3,4\}$.  Note that $M$  has label sequence  $ \lambda (M) = (4,3,1)$.
 Observe also that 
$\lambda (T_2(M)) = (4,1,2)$, that $\lambda (T_1T_2(M)) = (1,4,2)$ and that $\lambda (T_2T_1T_2(M)) = (1,2,3)$.  Thus, 
$T_2T_1T_2(M)$ is the unique ascending chain with respect to $\lambda $. 

 This is not equal to $T_1T_2T_1(M)$, as we show in our calculations given next.  Notice that 
$\lambda (T_1(M)) = 
(3,4,1)$, that $\lambda (T_2T_1(M)) = (3,1,2)$ and that $\lambda (T_1T_2T_1(M)) = (1,3,2)$.  Thus,  $\lambda (T_1T_2T_1(M)) \ne \lambda (T_2T_1T_2(M))$, implying
$T_1T_2T_2(M) \ne T_2T_1T_2(M)$.    In other words, our operators do not uphold all of the braid relations for the symmetric group.  

If we apply $T_2$ to $T_1T_2T_1(M)$, we do get a maximal chain with label sequence $(1,2,3)$.  Thus, 
$\lambda (T_2T_1T_2T_1(M)) = \lambda (T_2T_1T_2(M)) = (1,2,3)$.  Since there is a unique ascending chain with respect to $\lambda $, this implies
$T_2T_1T_2T_1(M) = T_2T_1T_2(M)$.
\end{ex}



Next we 
show 
 that any two maximal chains in  a geometric lattice $L$ of rank $r$  
 are within distance ${r\choose 2}$ of each other in the facet-ridge incidence graph. 

\begin{thm}
The diameter of the facet-ridge incidence graph of a geometric lattice $L$ of rank $r$ is at most ${r \choose 2}$
\end{thm}

\begin{proof}
Our plan is to show for any maximal chain $M$ in $L$ that that we can choose a minimal labeling that makes $M$ the unique ascending chain.  Once we do this, the result will follow from Corollary ~\ref{dist-to-ascending}.

Say that $M$ is the maximal chain $\hat{0} = x_0 \lessdot  x_1\lessdot  x_2\lessdot \cdots \lessdot x_{r-1}\lessdot x_r =  \hat{1}$.  Notice for $1\le i \le r$ that $A(x_i)\setminus A(x_{i-1})$ is nonempty since $L$ is atomic which means each $x_i$ is a join of atoms which must include one or more atoms  that are not less than or equal to $x_{i-1}$.  
Also notice by construction that 
$A(x_i)\setminus A(x_{i-1})$ has empty intersection with $A(x_j)\setminus A(x_{j-1})$ for $i\ne j$.  Finally, notice that $A(x_r)$ includes all of the atoms of $L$, implying that every atom belongs to $A(x_i)\setminus A(x_{i-1})$ for some $i$. 
Choose any total order on the set $A$ of  atoms of $L$  such that 
$a$ comes earlier than $a'$ for $a,a'\in A$ whenever $a\in A(x_i)\setminus A(x_{i-1})$ and $a'\in A(x_j)\setminus A(x_{j-1})$ for $i<j$.  Observe that the minimal labeling derived from any such total order on the atoms of $L$ will by construction have $M$ as an ascending chain.  By Theorem ~\ref{any-atom-ordering}, this minimal labeling will be an EL-labeling.  
\end{proof}


\begin{figure}
\begin{tikzpicture}
\node (bot) at (0,0) {$\emptyset$};

\node at (-0.67,1.5) (2) {$\{2\}$};
\node at (-2, 1.5) (1) {$\{1\}$};
\node at (0.67,1.5) (3) {$\{3\}$};
\node at (2,1.5) (4) {$\{4\}$};

\node[text=blue] at (-1.15,0.575) {$1$};
\node[text=red] at (1.15,0.575) {$4$};

\node at (-2.5,3) (124) {$\{1,2,4\}$};
\node at (-1,3) (13) {$\{1,3\}$};
\node at (1,3) (23) {$\{2,3\}$};
\node at (2.5,3) (34) {$\{3,4\}$};

\node[text=blue] at (-2.5,2.25) {$2$};
\node[text=red] at (2.5,2.25) {$3$};
\node[text=violet] at (0,2.425) {$1$};

\node at (0,4.5) (1234) {$\{1,2,3,4\}$};

\node[text=blue] at (-1.4,3.9) {$3$};
\node[text=red] at (1.4,3.9) {$1$};

\draw[ultra thick, blue]  (bot)--(1);
\draw (bot)--(2);
\draw (bot)--(3);
\draw[ultra thick, red] (bot)--(4);

\draw[ultra thick, blue]  (1)--(124);
\draw (2)--(124);
\draw[ultra thick, violet]  (4)--(124);
\draw (1)--(124);
\draw (1)--(13);
\draw (3)--(13);
\draw (2)--(23);
\draw (3)--(23);
\draw (3)--(34);
\draw[ultra thick, red] (4)--(34);

\draw[ultra thick, blue]  (124)--(1234);
\draw (13)--(1234);
\draw (23)--(1234);
\draw[ultra thick, red] (34)--(1234);
\end{tikzpicture}
\caption{A decreasing maximal in red at distance less than $\binom{r}{2}$ on the increasing maximal chain in blue.}
\label{fig:dec}
\end{figure}
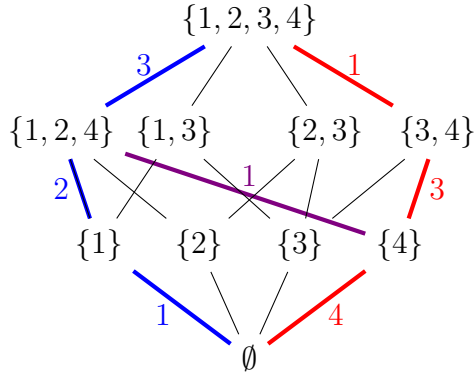

One could also  ask whether the facet associated to every  maximal chain with a decreasing label sequence has  distance exactly  $\binom{r}{2}$ from the facet given by the increasing chain.  This is not the case, as the next example illustrates.  

\begin{ex}\label{short-cuts}
Consider  the poset whose Hasse diagram is depicted in Figure~\ref{fig:dec}.
Here we have a poset of rank $r=3$.  The decreasing chain drawn in red is at a distance of $2$ from the increasing chain shown in blue, whereas ${r\choose 2} = 3$ in this case.
The one purple edge is used  in an intermediate step in moving from this decreasing chain to  the increasing chain.
The label sequences seen moving from this decreasing maximal chain to the increasing maximal chain are $(4,3,1)$, $(4,1,3)$, and $(1,2,3)$.
Notice in the last step the label sequence changes from $(4,1,3)$ to $(1,2,3)$, reducing the number of inversions by  $2$  in a single step.
\end{ex}


\begin{rk}
In fact, Example ~\ref{short-cuts} shows something more than discussed above.  It shows that there are adjacent vertices $u$ and $v$ in the facet-ridge incidence graph of a finite geometric lattice such that $G_{lex}$ does not include a directed edge from $u$ to $v$ but nonetheless the distance from $v$ to the sink vertex of $G_{lex}$  is strictly less than the distance from  $v$ to the sink in $G_{lex}$. 
 Thus, one may not a priori restrict attention to the  edges in 
 $G_{lex}$ for purpose of proving sharpness of our diameter  upper bound.  
\end{rk}

Now we turn to the question of sharpness of our  diameter upper bound. 

\begin{defn}
We say that an atom $a$ is  {\bf associated to a cover relation $u\lessdot v$}  whenever $a\in A(v)\setminus A(u)$.
\end{defn}
\begin{defn}
  We define an {\bf independent set} to be a set of atoms whose cardinality equals the rank of the join of this set of atoms.  This is equivalent to the notion of independent set  in the theory of matroids.  
\end{defn}

\begin{lem}\label{indep-atoms}
Consider any minimal labeling of any  geometric lattice $L$.  
Any set of atoms all associated to distinct cover relations in a maximal chain of $L$  is  an independent set.  
\end{lem}

\begin{proof}
Given a maximal chain $M$, let $S$ be a subset of the atoms of $L$ consisting of exactly one atom associated to each cover relation in $M$.  
Proceeding up the  maximal chain $M$, each step upward makes the rank go up by one.  Thus, 
each step upward in $M$ increases by one  the  rank  of the join of the subset of  $S$ consisting of the atoms associated to the  cover relations 
encountered so far.  Each such step also increases the size of the set of atoms chosen so far by one.  This shows the desired independence  result for  sets of atoms associated to 
distinct consecutive cover relations in a maximal chain.  We get the result for an arbitrary set of cover relations in a maximal chain from the fact from matroid theory that any subset of an independent set is independent.  
\end{proof}




\begin{lem}\label{iff-lemma} 
Given a minimal labeling $\lambda $  for a  geometric lattice $L$  of rank $r$, 
consider the label sequence $(\mu_1,\dots ,\mu_r)$ on any maximal chain.  The increasing rearrangement of $(\mu_1,\dots ,\mu_r)$ is the label sequence of a maximal 
chain in $L$ 
 if and only if this increasing rearrangement is  the lexicographically smallest 
 increasing rearrangement of any label sequence of any maximal chain in $L$.  
\end{lem}

\begin{proof}
Since $\lambda $ is  an EL-labeling, every interval has  exactly one ascending label sequence.   
By Lemma ~\ref{indep-atoms}, the unique ascending chain label sequence $(\mu_1,\dots ,\mu_r)$ consists of the elements of an independent set  listed in ascending order.
 
 Suppose the  ascending rearrangement $\nu_1 < \cdots <\nu_r$ of some other   maximal chain label sequence has   $(\nu_1,\dots ,\nu_r)$  strictly smaller in 
 lexicographic order  than $(\mu_1,\dots ,\mu_r)$.
   Then according to the definition of EL-labeling,  $(\nu_1,\dots ,\nu_r)$  
could not be the  label sequence of any maximal chain.    By Lemma ~\ref{indep-atoms}, $\{ \nu_1,\dots ,\nu_r\} $ is an independent set,  implying  that
$\hat{0}\lessdot \nu_1 \lessdot \nu_1\vee \nu_2 \lessdot \cdots \lessdot \nu_1\vee\cdots \vee \nu_r$ is a maximal
chain.  By definition of minimal  labeling, this maximal chain  has a label 
sequence $(\rho_1,\dots ,\rho_r)$ with  $\rho_i\le \nu_i$ for $i=1,2,\dots ,r$.  In particular, it has a label sequence
that is smaller lexicographically than  $(\mu_1,\dots ,\mu_r)$, 
a contradiction.  
 This shows that  the lexicographically first  maximal chain 
 label sequence is at least as small in lexicographic order  as  the  increasing rearrangement of the  label sequence of  every other maximal
 chain. 

Conversely, consider the  lexicographically smallest increasing rearrangement of any maximal chain label sequence.  By Lemma
~\ref{indep-atoms}, this  is a sequence $(\nu_1,\dots ,\nu_r)$ whose 
labels comprise an independent set.
 Note that 
$\nu_1\lessdot \nu_1\vee \nu_2 \lessdot \cdots \lessdot \nu_1\vee\cdots \vee\nu_r$  is a maximal chain;  moreover, it has label sequence
$(\nu_1,\dots ,\nu_r)$ unless it 
has a label sequence $(\rho_1,\dots ,\rho_r)$ with some 
$\rho_i < \nu_i$ and with $\rho_i\le \nu_i$ for $i=1,2,\dots ,n$.  
But given such a label sequence $(\rho_1,\dots ,\rho_r)$ not equalling $(\nu_1,\dots ,\nu_r)$, this would imply that  
the increasing rearrangement of $(\rho_1,\dots ,\rho_r)$ would  be lexicographically smaller  
than $(\nu_1,\dots ,\nu_r)$, 
contradicting our  lexicographic  minimality assumption for  the label sequence $(\mu_1,\dots ,\nu_r)$ on the unique ascending chain of $L$.
\end{proof}

\begin{lem}\label{rearranged-smallest} 
Given a   geometric lattice of rank $r$ with a minimal labeling, 
every possible rearrangement of the unique  maximal chain  label sequence that is ascending  
arises as the label sequence of  some maximal chain.
\end{lem}

\begin{proof}
Let $(\lambda_1,\dots ,\lambda_r)$ denote the label
sequence for the unique ascending chain.  Consider  the  label sequence $(\lambda_{\sigma (1)},\dots ,\lambda_{\sigma (r)})$ 
given by any  permutation $\sigma \in S_r$. 
By Lemma ~\ref{indep-atoms}, $\{ \lambda_1,\dots ,\lambda_r\} $ is an independent set.  Thus, 
$$\hat{0}\lessdot \lambda_{\sigma (1)} \lessdot \lambda_{\sigma (1)} \vee \lambda_{\sigma (2)}\lessdot \cdots \lessdot 
\lambda_{\sigma (1)} \vee\cdots \vee \lambda_{\sigma (r)} $$ is a maximal chain.
We will show that  $(\lambda_{\sigma (1)},\dots ,\lambda_{\sigma (r)})$  is the label sequence for this  maximal chain.  
By definition of a minimal
labeling, we know that every label  $\mu_j$  in the label sequence
 $(\mu_1,\dots ,\mu_r)$  for this maximal chain  must satisfy $\mu_j\le \lambda_{\sigma (j)}$.  From this it follows that the ascending 
rearrangement of this  label sequence  
would be strictly smaller than $(\lambda_1,\dots ,\lambda_r)$ in lexicographic order if we had $\mu_j < \lambda_{\sigma_j}$ for any  $j\in \{ 1,\dots ,r\} $. 
 However, $\{ \mu_1,\dots ,\mu_r\} $ 
 must also be an independent set of 
size $r$,  by Lemma \ref{indep-atoms}.    By Lemma ~\ref{iff-lemma}, having $\mu_j < \lambda_{\sigma_j}$ for some $j$ would therefore
contradict  $(\lambda_1,\dots ,\lambda_r)$ being  the unique ascending chain.  
\end{proof}

\begin{lem}\label{descent-swapping-lemma} 
Consider a   geometric lattice $L$  of rank $r$ with a minimal labeling $\lambda $.
Consider a maximal chain $M$ whose label sequence is a rearrangement of the label sequence on the unique ascending chain in $L$.  
Consider any $u_{i-1}\lessdot u_i \lessdot u_{i+1}$ in $M$ such that there is a  descent  $\lambda (u_{i-1},u_i) > \lambda (u_i,u_{i+1})$ in
the label sequence for $M$.    
  Then $(\lambda (u_i,u_{i+1}), \lambda (u_{i-1},u_i))$ 
is  the lexicographically smallest label
sequence on any saturated chain within  
$[u_{i-1},u_{i+1}]$.  
\end{lem}

\begin{proof}
Observe that $\lambda (u_i,u_{i+1}) \in A(u_{i-1} \vee \lambda (u_i,u_{i+1}))\setminus A(u_{i-1})$.  Also note that 
$\lambda (u_{i-1},u_i) \not\in A(u_{i-1} \vee \lambda (u_i,u_{i+1}))$ due to 
the union of  the set  $\{ \lambda (u_i,u_{i+1}), \lambda (u_{i-1},u_i)\} $ and the set of labels on 
any saturated chain from $\hat{0}$ to $u_{i-1}$ being an independent set, hence having join of strictly greater rank than 
$u_{i-1} \vee \lambda (u_i,u_{i+1})$.  Since $\lambda (u_{i-1},u_i)\in A(u_{i+1})$, we have 
$\lambda (u_{i-1},u_i)\in A(u_{i+1})\setminus A(u_{i-1} \vee \lambda (u_i,u_{i+1}))$.  

Equipped with these observations, 
we will show that 
$(\lambda (u_i,u_{i+1}), \lambda (u_{i-1},u_i))$ is the label sequence for the  
 saturated chain  $u_{i-1}  \lessdot u_{i-1} \vee \lambda (u_i,u_{i+1}) \lessdot u_{i+1}$
  in $[u_{i-1},u_{i+1}]$. Otherwise this saturated chain  would have a lexicographically 
smaller label sequence than $(\lambda (u_i,u_{i+1}),\lambda (u_{i-1},u_i))$.  
This  sequence  would  combine with the labels  on the  cover relations  of $M$ appearing below and above 
$[u_{i-1},u_{i+1}]$  to 
give a maximal chain whose  label sequence has an ascending rearrangement that  is strictly smaller than the label sequence of the  unique 
ascending chain of $L$.  By Lemma ~\ref{iff-lemma}, this gives a contradiction.  
\end{proof}



 

Next is a result providing evidence that our  upper bound  of ${r\choose 2}$ on the diameter of the facet-ridge incidence graph is sharp.


\begin{thm}
Consider any  geometric lattice $L$ of rank $r$.  Consider  the  minimal labeling $\lambda $ on $L$ induced by any choice of atom ordering for $L$, and 
consider   the  
directed graph $G_{lex}$ as in  Definition ~\ref{G-lex}   that results from the EL-labeling $\lambda $.  
$G_{lex}$ contains a vertex whose distance  within $G_{lex}$ 
 to the unique ascending chain   in $L$ 
 is exactly ${r\choose 2}$.  
\end{thm}

\begin{proof}
We have already proven in Corollary ~\ref{dist-to-ascending} 
 that the diameter of the facet-ridge incidence graph of any finite  geometric lattice of rank $r$  is at most ${r\choose 2}$.  Moreover, we did so by choosing any  maximal chain $C'$
 in $L$  to be the ascending chain and then exhibiting 
for every  maximal chain  $C$ in $L$    that there is a path of  length at most ${r\choose 2}$ from $C$ to $C'$ in the facet-ridge incidence graph.  
One may observe that this path appearing in 
the proof of Lemma ~\ref{assoc-an-expression} was constructed  in such a way that it   only used  directed 
edges that are present in $G_{lex}$.  This allows  us to deduce this same upper bound of ${r\choose 2}$  for the distance from any vertex in $G_{lex}$ to the unique sink in $G_{lex}$.
 What remains is to prove that  this bound  is sharp. 

 Let $\hat{0} = x_0 \lessdot x_1\lessdot x_2 \lessdot \cdots \lessdot x_{r-1}\lessdot x_r =  \hat{1}$ be our unique  ascending chain with respect to the minimal labeling $\lambda $ 
 given by our choice of atom ordering.
    Let $(\lambda_1,\dots ,\lambda_r)$ be the label sequence for this 
  ascending chain. 
  Lemma ~\ref{rearranged-smallest} implies that $(\lambda_r,\dots ,\lambda_1)$  is also  the label sequence of a maximal chain in $L$. 
  %
  Our plan is to  
  show that  each directed  edge in every  directed path  in  $G_{lex}$  from any  maximal  chain with label sequence 
  $(\lambda_r,\dots ,\lambda_1)$ to the ascending chain labeled $(\lambda_1,\dots ,\lambda_r)$ eliminates   exactly one  inversion from the label sequence. 
      The result will  then
  follow from the fact that $(\lambda_r,\dots ,\lambda_1)$ has exactly ${r\choose 2}$ inversions while $(\lambda_1,\dots ,\lambda_r)$ has 0 inversions.
  
  Each directed edge $M\rightarrow M'$ 
  replaces a descent $(j,i)$ for $j>i$ in the label sequence $\lambda (M)$ with an ascent of the form $(i,j')$ for some $j'>i$ in the label sequence $\lambda (M')$. 
     By virtue of how a minimal labeling is constructed, 
  we must have $j'\le j$, as we justify next.   Let 
  $u\lessdot v\lessdot w$ be the pair of cover relations in $M'$  with $\lambda (u,v) = i$ and $\lambda (v,w) = j'$, and let 
  $u\lessdot v'\lessdot w$ be the pair of cover relations in $M$ with $\lambda (u,v') = j$ and $\lambda (v',w)=i$.  Observe that 
  $i\not\in A(v')$  
  which implies 
  $j\not \in A(v)$ while  it is also true that $j\in A(w)$; thus, $j\in A(w)\setminus A(v)$ implying $j' = \min (A(w)\setminus A(v) ) \le j$.   This proves our claim.  
 
    If $j'=j$ for each directed edge in $G_{lex}$ under consideration, then
  the result follows from the fact that   
  swapping $i$ and $j$ eliminates  the $(j,i)$ inversion while preserving all other inversions.    
But  our choice of a  decreasing chain 
   having the same set of labels as the ascending chain  
   ensures we will have $j'=j$ at each directed edge  under consideration, 
   by Lemma ~\ref{descent-swapping-lemma}. 
 \end{proof}


\section{Connection to maximal chain descent orders}

Next we relate the graph $G_{lex}$ to a partial order known as the maximal chain descent order which was recently introduced by Stephen Lacina in \cite{SL}.

\begin{defn}
Consider any  finite bounded poset $P$ with an EL-labeling $\lambda $.  Define a partially ordered set  called 
the {\bf maximal chain descent order}  given by $P$ and $\lambda $,  denoted $P_{\lambda } (2)$,   as follows.  The elements of $P_{\lambda }(2)$
are the maximal chains in $P$.  The order relations for $P_{\lambda }(2)$ 
 are obtained by taking the transitive closure of the following order relations which we call ``polygon moves''.  We have a polygon move from  a maximal chain
 $v$ in $P$  to maximal chain  $u$ in $P$  whenever $u$ may be obtained from $v$ by deleting from $v$ a  single poset  element $x_i$  such that $v$  includes elements 
 $x_{i-1}\lessdot x_i \lessdot x_{i+1}$ with a descent at $x_i$,   and then 
replacing $x_i$ with the unique ascending chain  in $P$  from $x_{i-1}$ to $x_{i+1}$ with respect to $\lambda $.  
\end{defn}

\begin{prop}
Consider the  oriented subgraph  $G_{lex}$ of the facet-ridge incidence graph 
 of a  geometric lattice $L$  with respect to a minimal labeling $\lambda $.
Then $G_{lex}$ is exactly the Hasse diagram of the maximal chain descent order given by  $L$ and $\lambda $. 
\end{prop}

\begin{proof}
This follows directly from the result of Stephen Lacina \cite{SL}  that any finite geometric lattice  with EL-labeling given by a minimal labeling is polygon strong, hence polygon complete, implying that the so-called polygon moves are exactly the cover relations of the associated maximal chain descent order.
\end{proof}


\begin{thebibliography}{199}


\bibitem{BJ1} A. Bj\"orner, Shellable and Cohen-Macaulay partially ordered sets, {\it Trans. Amer. Math. Soc.} {\bf 260} (1980), 159--183.


\bibitem{bj} A. Bj\"orner, Orderings of Coxeter groups. Combinatorics and algebra (Boulder,
Colo., 1983), 175--195, \textit{Contemp. Math.}, \textbf{34}, Amer. Math. Soc., Providence, RI, 
1984.

\bibitem{Bj2}A. Bj\"orner, The homology and shellability of matroids and geometric lattices, in ``Matroid Applications'' (ed. N. White), Cambridge Univ. Press, 1992, pp. 226--283.

\bibitem{BB} A. Bj\"orner and F. Brenti, Combinatorics of Coxeter groups,
\textit{Graduate Texts in Mathematics}, \textbf{231}, Springer, New York, 2005.  xiv + 363 pp.





\bibitem{BW-on-lex} A. Bj\"orner and M. Wachs, On lexicographically shellable posets, 
\textit{Trans. Amer. Math. Soc.}, \textbf{277} (1983), no 1, 323--341.

\bibitem{CS} F. Criado and F. Santos, The maximum diameter of pure simplicial complexes and pseudo-manifolds, \textit{Discrete Comput. Geom.}, \textbf{58} (2017), no 3, 643--649.
  
  \bibitem{DH} R. Davidson and P. Hersh,  A lexicographic shellability characterization of geometric lattices, {\it J. Combin. Th. Ser. A}, {\bf 123} (2014), 8--13.


  \bibitem{SL} S. Lacina, Maximal chain descent orders: a tool for studying the structure of lexicographic shellings, preprint 2023, available at   arXiv:2209.15142.
  

 













\bibitem{Sa} F. Santos, A counterexample to the Hirsch conjecture, \textit{Annals of Mathematics (2)} 
{\bf 176} (2012), 383--412.



\bibitem{st} R. P. Stanley, Enumerative Combinatorics, Volume 1: Second Edition. Cambridge Studies in Advanced Mathematics, 49. Cambridge University Press, New York, 2012.



\end{thebibliography}
\end{document}